\documentclass[10pt]{article}

\usepackage{amsmath}
\usepackage{amsthm}
\usepackage{amssymb}
\usepackage[all]{xy}
\usepackage{latexsym}

\newcommand{\FF}{\mathbf{F}}

\newcommand{\ZZ}{\mathbf{Z}}
\newcommand{\CALA}{\mathcal{A}}

\newcommand{\CALS}{\mathcal{S}}

\newcommand{\intr}{\operatorname{Int}}

\newcommand{\aut}{\operatorname{Aut}}

\newtheorem{theorem}{Theorem}
\newtheorem{proposition}{Proposition}
\theoremstyle{definition}
\newtheorem{definition}{Definition}

\newtheorem{example}{Example}

\begin{document}

\title{Small connected quandles.}
\author{F.J.-B.J.~Clauwens}

\maketitle

In \cite{hendersonmn} an attempt was made to classify quandles of small order.
There appear to be $404$ quandles of order $5$, 
$6658$ quandles of order $6$,
$152900$ quandles of order $7$
and $5225916$ quandles of order $8$.
For order $9$ the search space was too large to finish the computation.
If however one restricts to the important subclass of connected quandles
then classification seems to be more accessible to computation,
comparable to the classification of groups of given order.
It is the purpose of this note to classify connected quandles up to order $14$,
and in particular to show that there is no connected quandle of order $14$.

\section{Introduction}

Note that our notation and terminology differs somewhat from 
our main reference \cite{andrug3}.
For this reason we repeat some definitions and proofs.

\begin{definition}
A quandle is a set $Q$ equipped with a binary operation $\star$
with the following properties:
\begin{itemize}
\item
$a\star a=a$ for all $a\in Q$.
\item
$(a\star b)\star c=(a\star c)\star(b\star c)$ for all $a,b,c\in Q$.
\item
For all $a,b\in Q$ there is a unique $c\in Q$ such that $c\star b=a$.
\end{itemize}
\end{definition}

\begin{example}
\label{ex1}
Let  be given a group $G$.
Let $Q$ be $G$ as a set and define $\star$ by
$a\star b=b^{-1}ab$.
Then $Q$ is a quandle.
\end{example}

\begin{definition}
A quandle homomorphism is a map $\rho$ of quandles satisfying $\rho(a\star b)=\rho(a)\star\rho(b)$ for all $a,b\in Q$.
\end{definition}

If $Q$ is a quandle and $a\in Q$ then the map $\phi_a\colon Q\to Q$
given by $\phi_a(b)=b\star a$ is quandle automorphism of $Q$.
In the group $\aut(Q)$ of all quandle automorphisms of $Q$
the $\phi_a$ generate a subgroup called the group $\intr(Q)$
of interior automorphisms.
This defines a map $\phi\colon Q\to\intr(Q)$,
which is a quandle homomorphism with respect
to the quandle structure on $\intr(Q)$ given in example \ref{ex1}.
If $\phi$  is injective then $Q$ is called faithful.

\begin{definition}
A quandle $Q$ is called connected if the action of $\intr(Q)$ on $Q$ is transitive.
\end{definition}

\begin{example}
Let $M$ be an abelian group and $T\colon M\to M$ be an automorphism.
Let $Q$ be  $M$ as a set, and define $\star$ by
$a\star b=T(a)+b-T(b)$.
Then $Q$ is a quandle which is connected iff $1-T$ is invertible.
\end{example}

\begin{proposition}
\label{abel}
If a quandle $P$ is connected and $G=\intr(P)$ is abelian then $P$ is trivial.
\end{proposition}

\begin{proof}
Let $a,b\in P$.
Then by connectednes $\phi_a$ and $\phi_b$ are in the same conjugacy class of $G$.
Since $G$ is abelian this means that $\phi_b=\phi_a$.
Therefore $a\star b=\phi_b(a)=\phi_a(a)=a$.
\end{proof}

\begin{definition}
(See \cite{andrug3}).
Let $Q$ be a quandle and $S$ be a set.
A dynamical cocycle on $Q$ with values in $S$ is a map
$\alpha\colon Q\times Q\times S\times S\to S$ with the following properties:
\begin{itemize}
\item
$\alpha_{x,x}(s,s)=s$ for all $x\in Q$ and $s\in S$.
\item
$\alpha_{x\star y,z}(\alpha_{x,y}(s,t),u)=
\alpha_{x\star z,y\star z}(\alpha_{x,z}(s,u),\alpha_{y,z}(t,u))$ for all $x,y,z\in Q$ and $s,t,u\in S$.
\item
For all $x,y\in Q$ and all $t\in S$ the map $s\mapsto\alpha_{x,y}(s,t)$ is a bijection $S\to S$.
\end{itemize}
In this case the set $Q\times S$ can be equipped with a quandle structure by the formula
$(x,s)\star(y,t)=(x\star y,\alpha_{x,y}(s,t))$.
\end{definition}

\begin{proposition}
If $Q$ is connected and $\rho\colon Q\to P$ is a surjective quandle homomorphism,
then $P$ is connected.
\end{proposition}

\begin{proof}
Let be given $a,b\in P$.
Choose $x,y\in Q$ such that $\rho(x)=a$ and $\rho(y)=b$.
Since $Q$ is connected there are $z_1,\dots,z_d\in Q$
and $e_1,\dots,e_d\in\{-1,+1\}$ such that
$\phi_{z_1}^{e_1}(\phi_{z_2}^{e_2}\dots(\phi_{z_d}^{e_d}(x))=y$.
But then
$\phi_{\rho(z_1)}^{e_1}(\phi_{\rho(z_2)}^{e_2}\dots(\phi_{\rho(z_d)}^{e_d}(a))=b$.
\end{proof}

\begin{proposition}
If $\rho\colon Q\to P$ is a surjective quandle homomorphism and $P$ is connected
then for $a,b\in P$ there is a bijection $\rho^{-1}(a)\cong\rho^{-1}(b)$.
In particular the cardinality of $P$ divides the cardinality of $Q$.
\end{proposition}

\begin{proof}
Since $P$ is connected we can choose $z_1,\dots,z_d\in P$
and $e_1,\dots,e_d\in\{-1,+1\}$ such that 
$\phi_{z_1}^{e_1}(\phi_{z_2}^{e_2}\dots(\phi_{z_d}^{e_d}(a))=b$.
Since $\rho$ is surjective we can choose $x_1,\dots,x_d\in Q$
such that $\rho(x_j)=z_j$ for all $j$.
Now $\phi_{x_1}^{e_1}\circ\phi_{x_2}^{e_2}\dots\circ\phi_{x_d}^{e_d}$
is a bijection from $\rho^{-1}(a)$ to $\rho^{-1}(b)$.
\end{proof}

\begin{proposition}
\label{dyn}
If $Q$ is connected and $\rho\colon Q\to P$ is a surjective quandle homomorphism,
then there exists a set $S$ and a dynamical cocycle $\alpha$
such that $Q\cong P\times_\alpha S$.
\end{proposition}

\begin{proof}
(See \cite{andrug3}).
Choose a set $S$ and bijections $g_p\colon\rho^{-1}(p)\to S$ for $p\in P$.
Define $\alpha\colon P\times P\times S\times S\to S$ by
$\alpha_{x,y}(s,t)=g_{x\star y}(g_x^{-1}(s)\star g_y^{-1}(t))$.
It is easy to check that $\alpha$ is a dynamical cocycle.
Define $\psi\colon Q\to P\times_\alpha S$ by
$\psi(q)=(\rho(q),g_{\rho(q)}(q))$.
It is easy to check that $\psi$ is a quandle isomorphism.
\end{proof}

A special kind of dynamical cocycles are those that do not depend
on the last entry:
\begin{definition}
(See \cite{andrug3}).
Let $Q$ be a quandle and let $H$ be a group.
A quandle cocycle on $Q$ with values in $H$ 
is a map $\beta\colon Q\times Q\to H$ such that
$\beta(x\star y,z)\beta(x,y)=\beta(x\star z,y\star z)\beta(x,z)$
and $\beta(x,x)=1$ for $x,y,z\in Q$.
\end{definition}

\begin{proposition}
Let $Q$ be a connected quandle 
and let $P$ be its image under $\phi\colon Q\to\intr(Q)$.
Then there exists a set $S$ and a quandle cocycle $\beta$
such that $Q\cong P\times_\beta S$.
\end{proposition}

\begin{proof}
See \cite{andrug3}, proposition 2.11.
\end{proof}

\begin{example}
\label{eight}
There exists a connected quandle $Q_8$ of order $8$
given by the following table:
\begin{equation*}
\begin{tabular}{|l|l|}
\hline
&$b$\\
\hline
$a$&$a\star b$\\
\hline
\end{tabular}
\qquad
\begin{tabular}{|l|l|l|l|l|l|l|l|l|l|l|l|l|}
\hline
&1&2&3&4&6&6&7&8\\
\hline
1&1&5&7&5&7&1&2&2\\
2&7&2&6&2&6&7&3&3\\
3&4&7&3&7&3&4&1&1\\
4&8&4&1&4&1&8&5&5\\
5&2&8&5&8&5&2&6&6\\
6&6&3&8&3&8&6&4&4\\
7&5&6&4&6&4&5&7&7\\
8&3&1&2&1&2&3&8&8\\
\hline
\end{tabular}
\end{equation*}
This quandle is not faithful:
the image $P$ under $\phi$ is the unique 
Alexander quandle $Q_4$ of order $4$.
Thus $Q_8$ is described by a quandle cocycle on $Q_4$.
Since $Q_8$ is not a product this cocycle is not a 
coboundary, which shows that Lemma 5.1 of \cite{grana3}
does not extend to $Q_4$.
\end{example}

\section{Degree $2$ extensions of Alexander quandles.}

The purpose of this section is to show that a
dynamical cocycle on a connected Alexander quandle 
with values in a set $S$ of cardinality $2$
is almost always a quandle cocycle.
For this purpose we identify $S$ with the field $\FF$ of two elements.

\begin{proposition}
The formula $\alpha_{x,y}(s,t)=\beta(x,y)+s+\gamma(x,y)t$
defines an equivalence between dynamical cocycles 
$\alpha\colon Q\times Q\times \FF\times\FF\to\FF$ 
and pairs of maps $\beta,\gamma\colon Q\times Q\to \FF$
satisfying the following conditions:
\begin{itemize}
\item
$\beta(x\star y,z)+\beta(x,y)=\beta(x,z)+\beta(x\star z,y\star z)+\gamma(x\star z,y\star z)\beta(y,z)$
\item
$\gamma(x,y)=\gamma(x\star z,y\star z)$.
\item
$\gamma(x\star y,z)=\gamma(x,z)+\gamma(x\star z,y\star z)\gamma(y,z)$.
\end{itemize}
\end{proposition}

\begin{proof}
A map $A\colon\FF\times\FF\to\FF$ such that
$s\mapsto A(s,t)$ is bijective for all $t$ is necessarily of the form $A=b+s+ct$.
The translation of the dynamical cocycle conditions on $\alpha$
to conditions on $\beta$ and $\gamma$ is straightforward.
\end{proof}

\begin{proposition}
Let $Q$ be the connected Alexander quandle
associated to the abelian group $M$ and the automorphism $T$.
Then there exists a map $\mu\colon M\to\FF$
such that $\gamma(x,y)=\mu(x-y)$.
It has the following properties:
\begin{itemize}
\item
$\mu(Tx)=\mu(x)$.
\item
$\mu(x+y)+\mu(x)\mu(y)=\mu(Tx+y)$.
\end{itemize}
\end{proposition}

\begin{proof}
The second condition on $\beta,\gamma$ now reads
\begin{equation*}
\gamma(x,y)=\gamma(Tx+z-Tz,Ty+z-Tz)
\end{equation*}
If one substitutes $z=0$ one finds $\gamma(Tx,Ty)=\gamma(x,y)$.
If one substitutes $z=-(1-T)^{-1}Ty$ one finds $\gamma(x,y)=\gamma(Tx-Ty,0)=\gamma(x-y,0)$.
Thus one can put $\mu(x)=\gamma(x,0)$.
The first property of $\mu$ is now immediate.
From the third condition on $\beta,\gamma$ putting $y=0$ one gets
\begin{equation*}
\mu(Tx-z)=\mu(x-z)+\mu(Tx)\mu(-z)
\end{equation*}
Using $\mu(Tx)=\mu(x)$ this proves the second property of $\mu$.
\end{proof}

\begin{proposition}
If $M$ is finite then the  set $N=\{x\in M\;;\;\mu(x)=0\}$ is closed under addition and under $T$.
\end{proposition}

\begin{proof}
Putting $x=y=0$ in the second property of $\mu$ we find $\mu(0)^2=0$, so $\mu(0)=0$.
Now putting $y=-x$ in the second property of $\mu$ we find $\mu(x)\mu(-x)=\mu(Tx-x)$.
Thus $T-1$ maps $N$ to $N$.
Since $N$ is assumed to be finite $T-1$ is in fact a bijection from $N$ to $N$.

Putting $y=z-x$ in the second property of $\mu$ we find 
$\mu(z)+\mu(x)\mu(z-x)=\mu(Tx-x+z)$.
Thus if $x,z\in N$ then $Tx-x+z\in N$.
Writing $w=Tx-x$ and using the fact that $T-1$ is invertible on $N$
this says that $w+z\in N$ if $w,z\in N$.
\end{proof}

\begin{proposition}
\label{tplus1not0}
If $M$ is finite and  $1+T$ and $2$ are invertible on $M$ then $N=M$.
\end{proposition}

\begin{proof}
Since $N$ is closed under addition, 
multiplication by $2$ maps $N$ to $N$.
Moreover it is injective and thus bijective.
Consider first the case that $x\not\in N$.
Then $2x\not\in N$,
and so $\mu(2x)=1=\mu(x)$.
Therefore $\mu(Tx+x)=\mu(2x)+\mu(x)^2=1+1^2=0$;
this means that $Tx+x\in N$.
Now consider the case that $x\in N$.
Then clearly also $Tx+x\in N$.
Therefore $T+1$ maps $M$ to $N$;
and since $T+1$ is invertible this means that $N=M$.
\end{proof}

\begin{proposition}
\label{tplus1is0}
If $M$ is finite and if $2$ and $3$ are invertible on $M$ and $T=-1$ then $N=M$.
\end{proposition}

\begin{proof}
For $T=-1$ the second property of $\mu$ reads
\begin{equation*}
\mu(x+y)+\mu(x)\mu(y)=\mu(-x+y)
\end{equation*}
Therefore $\mu(x+y)=\mu(-x+y)$ if $x\in N$.
Writing $w=-x+y$ this yields $\mu(w+2x)=\mu(w)$ for $x\in N$.
Since $2$ is invertible on $N$ this means that $\mu$ is well defined on $M/N$.

Now suppose that $\xi$ and $\eta$ are elements of $M/N$ 
such that $\xi$ and $\eta$ and $\xi+\eta$ are nonzero.
Then $\mu(\eta-\xi)=\mu(\xi+\eta)+\mu(\xi)\mu(\eta)=1+1\cdot1=0$ 
which means that $\eta=\xi $ in $M/N$.
Therefore if $\xi$ and $\eta$ and $\zeta$ and the zero class $N$
would be four different elements of $M/N$ 
then we would have $\xi+\eta=\xi+\zeta=\eta+\zeta=N$
which is a contradiction.
This means that $M/N$ has at most three elements.

Now suppose that $\xi$ is a nontrivial element of $M/N$.
Then so are $2\xi$ and $3\xi$ since $2$ and $3$ are invertible.
Together with $N$ this gives four different elements of $M/N$.
Contradiction.
\end{proof}

\section{Simple quandles.}

\begin{definition}
A quandle $Q$ is called simple if the
only surjective quandle homomorphisms on $Q$ 
have trivial image or are bijective.
\end{definition}

\begin{example}
A connected quandle $P$ of prime order $p$ is simple.
This is because the cardinality of $\phi(P)$ must divide $p$.
\end{example}

\begin{proposition}
\label{zg}
Let $P$ be a simple connected quandle  and let $G=\intr(P)$.
Then $\phi\colon P\to G$ is injective,
its image $C$ generates $G$ and is a conjugacy class,
and the center $Z(G)$ is trivial.
\end{proposition}

\begin{proof}
The map $\phi\colon P\to C$ is a surjective quandle homomorphism
so is trivial or bijective.
In the first case we would have $b\star a=\phi_a(b)=b$ for all $a,b\in P$,
which contradicts connectedness.
So we can identify $P$ with its image $C$.

If $a,b\in C$ then by connectedness there are $c_1,\dots,c_d\in C$
and $e_1,\dots,e_d\in\{+1,-1\}$ such that 
$b=c_d^{-e_d}\dots c_1^{-e_1}ac_1^{e_1}\dots c_d^{e_d}$
which means that $a$ and $b$ are in the same conjugacy class of $G$.
On the other hand suppose $b\in G$ is in the same conjugacy class as $a\in C$, say $b=a*g$ 
Then we can write $g=c_1^{e_1}\dots c_d^{e_d}$
for some $c_i\in C$ and $e_i\in\{+1,-1\}$ and we have
$b=c_d^{-e_d}\dots c_1^{-e_1}ac_1^{e_1}\dots c_d^{e_d}\in C$.
So $C$ is a conjugacy class.

If $\psi\in G=\intr(P)$ commutes with all elements of $G$ then it commutes with all $\phi_b$,
which means that $\phi_b=\psi\phi_b\psi^{-1}=\phi_{\psi(b)}$.
Since $\phi$ is injective that means that $b=\psi(b)$ for all $b$, so $\psi=1$.
\end{proof}

\begin{theorem}
If $Q$ is a connected quandle of order $2p$ with $p$ prime
then $Q$ is simple unless $p=2$ or $p=3$.
\end{theorem}

\begin{proof}
Suppose that $\rho\colon Q\to P$ is a nontrivial surjective quandle homomorphism.
Then the cardinality of $P$ divides that of $Q$ so it equals $2$ or $p$.
Since $P$ is connected its cardinality can not be $2$,
so it must be $p$.
Therefore $Q\cong P\times_\alpha S$ for some dynamical cocycle $\alpha$
by proposition \ref{dyn}.

By proposition \ref{zg} the conditions of theorem 4.1 of \cite{etingofgs} are satisfied.
Therefore $\intr(P)$ is a subgroup of the affine group of $M=\ZZ/(p)$
which is another way of saying that $P$ is an Alexander quandle 
associated to $M$ and some $T$.
The automorphism $T$ is multiplication by some $w\in\ZZ/(p)$ with $w\not=0,1$.

Assume that  $p\not=2,3$.
If $w\not=-1$ then proposition \ref{tplus1not0} applies,
and if $w=-1$ then proposition \ref{tplus1is0} applies.
In both cases $N=M$ which means that $\mu=0$ 
and therefore $\gamma=0$.
Therefore $\alpha=\beta$, which is quandle cocycle.
By Lemma 5.1.  of \cite{grana3} this cocycle is a coboundary.
In other words $\beta$ can be assumed to vanish, 
and $Q\cong P\times S$.
This contradicts the connectedness of $Q$.
\end{proof}

\begin{example}
\label{six}
The $2$-cycles in $\CALS_4$ form a connected quandle $Q_{6,2}$ of order $6$.
The $4$-cycles in $\CALS_4$ form another connected quandle $Q_{6,4}$ of order $6$.
Both have a surjective homomorphism to the Alexander quandle $Q_3$
of order $3$, hence are not simple.
\end{example}

\begin{definition}
A group $G$ is called image-cyclic if every 
surjective homomorphism on $G$ has cyclic image or is bijective.
\end{definition}

\begin{proposition}
\label{imcyc}
Let $Q$ be a connected quandle.
Then $Q$ is simple if and only if $Q$ is faithful and $G=\intr(Q)$
is image-cyclic.
\end{proposition}

\begin{proof}
See \cite{andrug3}.
Let $Q$ be simple, so that proposition \ref{zg} applies.
Let $\psi\colon G\to H$ be a surjective group homomorphism.
Then its restriction $\rho\colon C\to\psi(C)$ is a surjective quandle homomorphism.
Therefore $\rho$ is bijective or trivial.
In the first case if $g\in G$ is in the kernel of $\psi$ then 
for $c\in C$ one has $\rho(g^{-1}cg)=\psi(g)^{-1}\psi(c)\psi(g)=\psi(c)=\rho(c)$
and therefore $g^{-1}cg=c$.
Since $C$ generates $G$ this means that $g\in Z(G)=\{1\}$,
so $\psi$ is bijective.
In the second case $\psi(C)$ is a point,
and this point generates $\psi(G)$ 
which is therefore cyclic.

Now assume that $Q$ is faithful and image-cyclic.
Suppose that $\rho\colon Q\to P$ is a surjective quandle homomorphism.
Then $\rho$ induces a surjective group homomorphism
$\psi\colon G=\intr(Q)\to H=\intr(P)$
by proposition 2.31 of \cite{eisermann5}.
It follows easily that $P$ is also faithful,
so that $\rho$ can be seen as restriction of $\psi$.
By hypothesis $\psi$ is bijective or $H$ is cyclic.
In the first case $\rho$ is injective.
In the second case $P$ is trivial by proposition \ref{abel}.
\end{proof}

In order to find simple quandles 
the following theorem of R.~Guralnick from \cite{andrug3} is essential:
\begin{theorem}
Let $G$ be an image-cyclic group with $Z(G)=1$.
Then there exist a simple group $L$, a positive integer $t$, 
and a cyclic subgroup $C=\langle\psi\rangle$ of $\aut(N)$, 
where $N=L^t$, such that one of the following holds:
\begin{itemize}
\item
$L$ is abelian, and $\psi$ acts irreducibly on $N$.
Moreover $G$ is the semidirect product $N\rtimes C$ of $N$ and $C$.
\item
$L$ is simple non-abelian,
$G=NC\cong(N\rtimes C)Z(N\rtimes C)$
and $\psi$ acts by $\psi(\ell_1,\dots,\ell_t)=(\theta(\ell_t),\ell_1,\dots,\ell_{t-1})$
for some $\theta\in\aut(L)$.
\end{itemize}
Conversely all such groups have the desired properties.
The group $L$, the number $t$ and the class of $\psi$ in $\aut(N)/\intr(N)$
are uniquely determined by $G$.
\end{theorem}

The simple quandles are now found as conjugacy classes
of groups of the above kind which are generating.

\begin{example}
The conjugacy class of a $3$-cycle in $\CALA_4$ forms a connected quandle $Q_4$ of order $4$.
Here $G=\CALA_4$ is image-cyclic of the first type, 
with $L$ of order $2$ and $t=2$ and $C$ of order $3$.
This quandle is the Alexander quandle with $M$ the field of $4$ elements
and $T$ multiplication with $w\not=0,1$.
\end{example}

\begin{example}
\label{ten}
The $2$-cycles in $\CALS_5$  form a connected quandle $Q_{10}$ of order $10$.
Here $G=\CALS_5$ is image-cyclic of the second type, 
with $L=\CALA_5$ and $t=1$ and $C$ of order $2$.
\end{example}

\begin{example}
\label{twelve}
The conjugacy class of a $5$-cycle in $\CALA_5$ forms a connected quandle $Q_{12}$ of order $12$.
Here $G=\CALA_5$ is image-cyclic of the second type, 
with $L=\CALA_5$ and $t=1$ and $C$ trivial.
\end{example}

In the first case of the theorem
a conjugacy $C$ class of $G$ which generates $G$
is an Alexander quandle associated to $M=N$ and $T=\psi$.

\begin{example}
Consider the Alexander quandle $Q$ associated to $M=\ZZ/(15)$
and $T$ given by $Tx=2x$.
Then $M$ has a $T$-invariant subgroup $3M$,
and this gives rise to a normal subgroup $H$ of $G=\intr(Q)$.
Since $G/H$ is not cyclic $G$ is not image-cyclic.
So not all Alexander quandles  are explained by this theorem
\end{example}

It is rather easy to construct Alexander quandles of a given order.
Thus if we are looking for non-obvious connected quandles 
we may concentrate on the second case of the theorem.
In many cases the order of a conjugacy class is much too large:

\begin{proposition}
If $t>1$ then the cardinality of a generating conjugacy class
is at least the order of $L$.
\end{proposition}

\begin{proof}
The group $G$ is generated by $L^t$ and an element $\sigma\in G$
such that conjugation with $\sigma$ realizes the action of $\psi$:
\begin{equation*}
\sigma(\ell_1,\dots,\ell_t)\sigma^{-1}=(\theta(\ell_t),\ell_1,\dots,\ell_{t-1})
\end{equation*}
Let $Q$ be generating conjugacy class and let $g\in Q$.
For certain $\ell_i$ and $j$ we can write $g=(\ell_1,\dots,\ell_t)\sigma^j$.
If $j$ would be a multiple of $t$ then all elements of $G$
would have the form $(z_1,\dots,z_t)\sigma^{mt}$,
and in particular $\sigma\in L^t\sigma^{mt}$ for certain $m$.
This easily leads to a contradiction.
Write $j=qt+r$ with $0<r<t$.
If we conjugate by $h=(k_1,\dots,k_t)\sigma$ we find
\begin{equation*}
\begin{split}
h^{-1}gh
&=\sigma^{-1}(k_1^{-1},\dots,k_t^{-1})(\ell_1,\dots,\ell_t)\sigma^j(k_1,\dots,k_t)\sigma\\
&=(\dots,\theta^{-1}(k_1^{-1}\ell_1)\theta^q(k_{t-r+1}))\sigma^j
\end{split}
\end{equation*}
Since $k_1$ and $k_{t-r+1}$ can be chosen independently,
the last coordinate can take any value in $L$.
\end{proof}

Thus if we are looking for non-Alexander simple quandles of order less than $60$
we can restrict our attention to the case $t=1$,
since the order of any nonabelian simple group is at least $60$.

Now $L$ is a subgroup of $N=G=\intr(Q)$ 
which is a subgroup of the group of all permutations of $Q$.
Thus if we are looking for such quandles of order at most $14$
we can restrict attention to finite simple groups $L$
which are subgroups of $\CALS_{14}$.
In particular the order should divide  $14!=2^11\cdot3^5\cdot5^2\cdot7^2\cdot11\cdot13$.
This already excludes most of the simple groups from the classification,
and the remaining ones can be found in the Atlas \cite{atlas},
together with information about subgroups.
Inspection of this infomation reveals that only $L=\CALA_5$ 
survives and proves the following theorem:

\begin{theorem}
The non-Alexander simple connected quandles of order $\leq 14$
are $Q_{10}$ and $Q_{12}$  discussed in example \ref{ten} and example \ref{twelve}.
\end{theorem}

In particular there are no such quandles of order $14$.
Since an abelian group $M$ of order $14$
can not have an automorphism $T$ such that $1-T$ is invertible
there are no Alexander quandles of order $14$ either.
Therefore it follows from our theorems that there are no connected quandles
of order $14$,
and only one of order $10$.

\section{Classification.}
As part of his Bachelor Thesis my student Julius Witte has made
a computer search for connected quandles of order $n\leq12$.
Since all simple quandles are found by the method of the last section,
one can restrict attention to quandles $Q$ which have a surjective
quandle homomorphism to a connected quandle $P$ of order $m$ dividing $n$.
In that case the group $G=\intr(Q)$ is not only a subgroup of
of $\CALS_n$ but in fact of the wreath product $\CALS_{n/m}^m\rtimes\CALS_m$; 
this group is usually much smaller.
The program takes as input a subgroup $\Gamma$ of $\CALS_n$
and tries to find all connected quandles $Q$ with $G<\Gamma$.

By \cite{grana3} all connected quandles of prime square order
are Alexander quandles.
So in the the selected range only orders $6$, $8$ and $12$
have to be investigated.
There are  two connected quandles $Q_{6,2}$ and $Q_{6,4}$ 
of order $6$ which have been described in example \ref{six}.
There are three connected quandles of order $8$.
One is $Q_8$ as  described in example \ref{eight},
and the other two are Alexander quandles associated 
to multiplication maps on the field of $8$ elements.

There are ten connected quandles of order $12$.
One is the simple quandle $Q_{12}$  described in example \ref{twelve}.
We describe the other nine by tables.
\vfill\eject

\noindent
{\bf Case 1.}
Faithful.
The order of $\intr(Q)$ is $216$.
It surjects to $Q_4$.
\begin{equation*}
\begin{tabular}{|l|l|l|l|l|l|l|l|l|l|l|l|l|}
\hline
&1&2&3&4&5&6&7&8&9&10&11&12\\
\hline
1&1&1&1&12&11&10&5&4&6&9&7&8\\
2&2&2&2&11&10&12&6&5&4&8&9&7\\
3&3&3&3&10&12&11&4&6&5&7&8&9\\
4&8&9&7&4&4&4&10&12&11&3&2&1\\
5&7&8&9&5&5&5&11&10&12&2&1&3\\
6&9&7&8&6&6&6&12&11&10&1&3&2\\
7&11&12&10&3&1&2&7&7&7&4&5&6\\
8&12&10&11&1&2&3&8&8&8&5&6&4\\
9&10&11&12&2&3&1&9&9&9&6&4&5\\
10&6&5&4&7&8&9&3&2&1&10&10&10\\
11&5&4&6&9&7&8&1&3&2&11&11&11\\
12&4&6&5&8&9&7&2&1&3&12&12&12\\
\hline
\end{tabular}
\end{equation*}

\noindent
{\bf Case 2.}
Faithful.
The order of $\intr(Q)$ is $96$.
It surjects to $Q_{6,2}$.
\begin{equation*}
\begin{tabular}{|l|l|l|l|l|l|l|l|l|l|l|l|l|}
\hline
&1&2&3&4&5&6&7&8&9&10&11&12\\
\hline
1&1&1&1&1&9&11&10&12&5&7&6&8\\
2&2&2&2&2&12&10&11&9&8&6&7&5\\
3&3&3&3&3&10&12&9&11&7&5&8&6\\
4&4&4&4&4&11&9&12&10&6&8&5&7\\
5&9&12&10&11&5&5&5&5&1&3&4&2\\
6&11&10&12&9&6&6&6&6&4&2&1&3\\
7&10&11&9&12&7&7&7&7&3&1&2&4\\
8&12&9&11&10&8&8&8&8&2&4&3&1\\
9&5&8&7&6&1&4&3&2&9&9&9&9\\
10&7&6&5&8&3&2&1&4&10&10&10&10\\
11&6&7&8&5&4&1&2&3&11&11&11&11\\
12&8&5&6&7&2&3&4&1&12&12&12&12\\
\hline
\end{tabular}
\end{equation*}

\noindent
{\bf Case 3.}
Not faithful: $\phi(Q)$ has order $6$.
The order of $\intr(Q)$ is $24$.
It surjects to $Q_{6,2}$.
\begin{equation*}
\begin{tabular}{|l|l|l|l|l|l|l|l|l|l|l|l|l|}
\hline
&1&2&3&4&5&6&7&8&9&10&11&12\\
\hline
1&1&3&1&3&9&9&11&11&6&7&7&6\\
2&4&2&4&2&10&10&12&12&7&6&6&7\\
3&3&1&3&1&12&12&10&10&5&8&8&5\\
4&2&4&2&4&11&11&9&9&8&5&5&8\\
5&12&11&12&11&5&5&6&6&3&4&4&3\\
6&9&10&9&10&6&6&5&5&1&2&2&1\\
7&11&12&11&12&8&8&7&7&2&1&1&2\\
8&10&9&10&9&7&7&8&8&4&3&3&4\\
9&6&8&6&8&1&1&4&4&9&12&12&9\\
10&8&6&8&6&2&2&3&3&11&10&10&11\\
11&7&5&7&5&4&4&1&1&10&11&11&10\\
12&5&7&5&7&3&3&2&2&12&9&9&12\\
\hline
\end{tabular}
\end{equation*}

\noindent
{\bf Case 4.}
Faithful.
The order of $\intr(Q)$ is $96$.
It surjects to $Q_{6,2}$.
\begin{equation*}
\begin{tabular}{|l|l|l|l|l|l|l|l|l|l|l|l|l|}
\hline
&1&2&3&4&5&6&7&8&9&10&11&12\\
\hline
1&1&3&1&3&9&12&11&10&5&8&7&6\\
2&4&2&4&2&11&10&12&9&8&6&5&7\\
3&3&1&3&1&12&9&10&11&6&7&8&5\\
4&2&4&2&4&10&11&9&12&7&5&6&8\\
5&9&11&12&10&5&5&6&6&1&4&2&3\\
6&12&10&9&11&6&6&5&5&3&2&4&1\\
7&11&12&10&9&8&8&7&7&4&3&1&2\\
8&10&9&11&12&7&7&8&8&2&1&3&4\\
9&5&8&6&7&1&3&4&2&9&12&12&9\\
10&8&6&7&5&4&2&3&1&11&10&10&11\\
11&7&5&8&6&2&4&1&3&10&11&11&10\\
12&6&7&5&8&3&1&2&4&12&9&9&12\\
\hline
\end{tabular}
\end{equation*}

\noindent
{\bf Case 5.}
Faithful.
The order of $\intr(Q)$ is $96$.
It surjects to $Q_{6,2}$ and to $Q_{6,4}$.
\begin{equation*}
\begin{tabular}{|l|l|l|l|l|l|l|l|l|l|l|l|l|}
\hline
&1&2&3&4&5&6&7&8&9&10&11&12\\
\hline
1&1&1&1&1&12&10&11&9&5&7&6&8\\
2&2&2&2&2&10&12&9&11&6&8&5&7\\
3&3&3&3&3&11&9&12&10&7&5&8&6\\
4&4&4&4&4&9&11&10&12&8&6&7&5\\
5&9&11&10&12&5&5&5&5&4&2&3&1\\
6&11&9&12&10&6&6&6&6&3&1&4&2\\
7&10&12&9&11&7&7&7&7&2&4&1&3\\
8&12&10&11&9&8&8&8&8&1&3&2&4\\
9&8&7&6&5&1&2&3&4&9&9&9&9\\
10&6&5&8&7&3&4&1&2&10&10&10&10\\
11&7&8&5&6&2&1&4&3&11&11&11&11\\
12&5&6&7&8&4&3&2&1&12&12&12&12\\
\hline
\end{tabular}
\end{equation*}

\noindent
{\bf Case 6.}
Not faithful: $\phi(Q)$ has order $6$.
The order of $\intr(Q)$ is $48$.
It surjects to $Q_{6,4}$.
\begin{equation*}
\begin{tabular}{|l|l|l|l|l|l|l|l|l|l|l|l|l|}
\hline
&1&2&3&4&5&6&7&8&9&10&11&12\\
\hline
1&1&3&1&3&10&9&10&9&5&6&6&5\\
2&4&2&4&2&9&10&9&10&8&7&7&8\\
3&3&1&3&1&11&12&11&12&7&8&8&7\\
4&2&4&2&4&12&11&12&11&6&5&5&6\\
5&9&11&9&11&5&7&5&7&2&3&3&2\\
6&10&12&10&12&8&6&8&6&3&2&2&3\\
7&12&10&12&10&7&5&7&5&4&1&1&4\\
8&11&9&11&9&6&8&6&8&1&4&4&1\\
9&6&7&6&7&3&4&3&4&9&12&12&9\\
10&5&8&5&8&4&3&4&3&11&10&10&11\\
11&7&6&7&6&2&1&2&1&10&11&11&10\\
12&8&5&8&5&1&2&1&2&12&9&9&12\\
\hline
\end{tabular}
\end{equation*}

\noindent
{\bf Case 7.}
Faithful.
The order of $\intr(Q)$ is $96$.
It surjects to $Q_{6,2}$.
\begin{equation*}
\begin{tabular}{|l|l|l|l|l|l|l|l|l|l|l|l|l|}
\hline
&1&2&3&4&5&6&7&8&9&10&11&12\\
\hline
1&1&1&2&2&9&10&11&12&8&7&6&5\\
2&2&2&1&1&12&11&10&9&5&6&7&8\\
3&4&4&3&3&10&12&9&11&6&8&5&7\\
4&3&3&4&4&11&9&12&10&7&5&8&6\\
5&12&9&11&10&5&8&8&5&1&3&4&2\\
6&11&10&9&12&7&6&6&7&4&1&2&3\\
7&10&11&12&9&6&7&7&6&3&2&1&4\\
8&9&12&10&11&8&5&5&8&2&4&3&1\\
9&5&8&7&6&2&3&4&1&9&12&12&9\\
10&6&7&5&8&4&2&1&3&11&10&10&11\\
11&7&6&8&5&3&1&2&4&10&11&11&10\\
12&8&5&6&7&1&4&3&2&12&9&9&12\\
\hline
\end{tabular}
\end{equation*}

\noindent
{\bf Case 8.}
The unique Alexander quandle of order $12$,
a product $Q_3$ and $Q_4$.
\begin{equation*}
\begin{tabular}{|l|l|l|l|l|l|l|l|l|l|l|l|l|}
\hline
&1&2&3&4&5&6&7&8&9&10&11&12\\
\hline
1&1&4&2&3&9&12&10&11&8&5&7&6\\
2&3&2&4&1&10&11&9&12&7&6&8&5\\
3&4&1&3&2&11&10&12&9&6&7&5&8\\
4&2&3&1&4&12&9&11&10&5&8&6&7\\
5&10&12&11&9&5&8&6&7&1&2&3&4\\
6&12&10&9&11&7&6&8&5&4&3&2&1\\
7&11&9&10&12&8&5&7&6&2&1&4&3\\
8&9&11&12&10&6&7&5&8&3&4&1&2\\
9&5&7&8&6&4&3&2&1&9&12&10&11\\
10&7&5&6&8&1&2&3&4&11&10&12&9\\
11&8&6&5&7&3&4&1&2&12&9&11&10\\
12&6&8&7&5&2&1&4&3&10&11&9&12\\
\hline
\end{tabular}
\end{equation*}

\noindent
{\bf Case 9.}
Faithful.
The order of $\intr(Q)$ is $96$.
It surjects to $Q_{6,4}$.
\begin{equation*}
\begin{tabular}{|l|l|l|l|l|l|l|l|l|l|l|l|l|}
\hline
&1&2&3&4&5&6&7&8&9&10&11&12\\
\hline
1&1&3&1&3&12&10&9&11&8&7&5&6\\
2&4&2&4&2&10&9&11&12&7&6&8&5\\
3&3&1&3&1&9&11&12&10&6&5&7&8\\
4&2&4&2&4&11&12&10&9&5&8&6&7\\
5&11&12&10&9&5&7&5&7&3&2&4&1\\
6&12&10&9&11&8&6&8&6&2&1&3&4\\
7&10&9&11&12&7&5&7&5&1&4&2&3\\
8&9&11&12&10&6&8&6&8&4&3&1&2\\
9&7&6&5&8&4&3&2&1&9&12&12&9\\
10&6&5&8&7&3&2&1&4&11&10&10&11\\
11&8&7&6&5&1&4&3&2&10&11&11&10\\
12&5&8&7&6&2&1&4&3&12&9&9&12\\
\hline
\end{tabular}
\end{equation*}

\vfill\eject

\end{document}